\documentclass[11pt,a4paper]{amsart}

\usepackage{amssymb, amsmath, amsthm, amscd}

\usepackage{enumerate}

\usepackage{mathrsfs}

\usepackage{tikz-cd}

\theoremstyle{plain} \newtheorem{theorem}{Theorem}[section]
\theoremstyle{plain} \newtheorem{corollary}[theorem]{Corollary}
\theoremstyle{plain} \newtheorem{proposition}[theorem]{Proposition}
\theoremstyle{plain}\newtheorem{lemma}[theorem]{Lemma}
\theoremstyle{definition} \newtheorem{definition}[theorem]{Definition}
\theoremstyle{definition}
\theoremstyle{remark}\newtheorem{remark}[theorem]{Remark}
\theoremstyle{definition}
\theoremstyle{remark}
\theoremstyle{definition}
\theoremstyle{definition}\newtheorem{problem}[theorem]{Problem}

\newcommand{\C}{{\mathbb{C}}}

\newcommand{\X}{\mathcal X}
\newcommand{\T}{\mathscr{T}(M)}
\newcommand{\D}{\text{\rm Diff}^0(M)}
\newcommand{\I}{\mathcal I(M)}
\newcommand{\Azero}{\text{\rm Aut}^0(X_0)}
\newcommand{\Aun}{\text{\rm Aut}^1(X_0)}
\newcommand{\A}{\text{\rm Aut}(X_0)}

\numberwithin{equation}{section}

\usepackage{hyperref}
\hypersetup{colorlinks = true, linkcolor=blue}

\title[Kuranishi and Teichm\"uller]{Kuranishi and Teichm\"uller}
\author{Laurent Meersseman}
\date{\today}

\subjclass{32G05, 
58H05, 
14D23 
}

\address{Laurent Meersseman\\
	Univ Angers, CNRS, LAREMA, SFR MATHSTIC\\
	F-49000 Angers, France\\ laurent.meersseman@univ-angers.fr}

\begin{document}
\begin{abstract}
The goal of this short article is to describe the local structure of the Teichm\"uller stack of \cite{LMStacks} in the neighborhood of a K\"ahler point. In particular we show that at a generic K\"ahler point $X$, Catanese Kur=Teich question, when interpretated at the level of stacks, has an affirmative answer. The situation may be much more complicated if $X$ is non-K\"ahler suggesting that Teichm\"uller spaces/stacks of non-K\"ahler manifold has a much richer geometry. 
\end{abstract}

\maketitle

\section{Introduction.}
\label{intro}

Let $X_0$ be a compact complex manifold with underlying $C^\infty$ manifold denoted by $M$. There are traditionnally two ways of describing the complex structures {\slshape near} $X_0$. From the one hand, one may (try to) construct a moduli space of complex structures on $M$ and look at a neighborhood of the class of $X_0$ in this moduli space. From the other hand, one may focus on small deformations of $X_0$ and look for a deformation from which all other deformations can be obtained by pull-back, after restriction to an adequate neighborhood of the base point.

Teichm\"uller refers to the Teichm\"uller space, that is the set of classes of complex manifolds diffeomorphic to $M$ up to biholomorphism smoothly isotopic to the identity; hence to the first setting. Kuranishi refers to the Kuranishi semi-universal deformation; hence to the second setting.

In complex dimension one, the Teichm\"uller space has a natural structure of a complex manifold and the base of the Kuranishi deformation of $X_0$ is a (germ of) neighborhood of the class of $X_0$ in the Teichm\"uller space. So the link between the two notions is direct.

In complex dimension strictly greater than one, the Teichm\"uller space is just a topological space, usually non Hausdorff and non locally Hausdorff (see \cite{LMStacks}, Examples 13.3 and 13.6); and there exists a surjective continuous mapping from the ($\C$-analytic) base of the Kuranishi family onto a neighborhood of the class of $X_0$ in the Teichm\"uller space. Catanese asked in \cite{Catsurvey}, see also \cite{Catsurvey2}, for conditions under which these two spaces are locally homeomorphic and proved it is the case in several cases including the case with trivial canonical bundle. However, due to usual non-local Hausdorffness, the existence of such a local homeomorphism is a rather restrictive property. And a useful criterion is missing. We use the slogan Kur=Teich to refer to this question and to a positive anwer to it.

In \cite{LMStacks}, we replace the Teichm\"uller space with the Teichm\"uller stack, an analytic Artin stack over the category of $\C$-analytic spaces (see also \cite{Cortona} for a comprehensive presentation). On the way of exhibiting an atlas for this stack, we also construct Kuranishi stacks which are, roughly speaking, the quotient of Kuranishi base spaces by the automorphism group of the base points.

Catanese's question Kur=Teich becomes in this new framework to find conditions for these two stacks to be locally isomorphic. It is now more natural since we stick to the analytic world. 

The goal of this short article is to show that, in the K\"ahler setting, we have Kur=Teich at a generic point of the Teichm\"uller space. We obtain this result as a consequence of a much more detailed statement. Indeed, we give a complete description of the Teichm\"uller stack around a point encoding a K\"ahler structure by showing in Theorem \ref{finitethm} that the natural inclusion of stacks $\text{Kur}\subset\text{Teich}$ is a finite \'etale analytic morphism. We then prove that Kur=Teich outside a strict analytic substack of Teich that we characterize in Theorem \ref{2ndmainthm}. 

We point out that such a result does not hold at the level of the Riemann moduli stack since the mapping class group of a, say projective, manifold can act on the Teichm\"uller stack with dense orbits (this is the case for $2$-dimensional tori \cite{K-S1} or for Hyperk\"ahler manifolds \cite{VerbitskyErgodic}), hence the inclusion of Kur in this moduli stack may be far from being a finite morphism. 

Moreover, at a non-K\"ahler point, the situation may be much more complicated and the finiteness property is also lost. This is only a theoretical statement and we unfortunately lack of examples. Indeed, we do not know of a single example with infinite fibers. However, our results and methods strongly suggest that they should exist and point towards a dichotomy between points with the above mentionned inclusion being a finite \'etale morphism (including but not equal to K\"ahler ones) and points with infinite fibers (which have to be non-K\"ahler). As a consequence, the local structure of the Teichm\"uller stack is much more singular in a sense at a non-K\"ahler point. If correct, this would really be surprising since, at the level of the Kuranishi space (and Kuranishi stack), there is no difference between K\"ahler and non-K\"ahler manifolds: the Kuranishi space of a K\"ahler, even of a projective, manifold can exhibit all the pathologies (for example not irreducible \cite{Horikawa}, not reduced \cite{Mumford}, arbitrary singularities \cite{Vakil}) the Kuranishi space of a non-K\"ahler one can have. In the same way, as noted above, there is no difference between them at the level of the Riemann moduli stack. {\sl This difference only appear when considering the Teichm\"uller stack} and suggests that the full complexity of the Teichm\"uller stack is only seen at non-K\"ahler points {\sl hence that its geometry cannot be fully understood without dealing with non-K\"ahler manifolds}.

We also prove two related results in the paper. Firstly, we show in Theorem \ref{thmuniversal} that the germ of Kuranishi stack at a point has a universal property. This generalizes the semi-universality property of the Kuranishi space. At a rough level, this is folklore (see for example \cite{VerbitskySurvey}), but we never saw a precise statement of such a property, probably because the stack setting developed in \cite{LMStacks} is necessary to a clear formulation. Secondly, we investigate when the local Teichm\"uller stack is an orbifold. We solve this question in Theorem \ref{mainorbifoldthm} in the K\"ahler setting.

\section{The Teichm\"uller Stack: basic facts}
\label{stack}
We recollect some facts about the Teichm\"uller stack of a connected, compact oriented $C^\infty$ manifold $M$ admitting complex structures. We refer to \cite{LMStacks} for more details.

Let $\mathfrak S$ be the category of analytic spaces and morphisms endowed with the transcendantal topology. Given $S\in \mathfrak S$, we call {\it $M$-deformation over $S$} a proper and smooth morphism $\mathcal X\to S$ whose fibers are compact complex manifolds diffeomorphic to $M$. As $C^\infty$-object, such a deformation is a bundle over $S$ with fiber $M$ and structural group $\text{Diff}^+(M)$ (diffeomorphisms of $M$ that preserve its orientation). It is called {\it reduced} if the structural group is reduced to $\text{Diff}^0(M)$. In the same way, a morphism of reduced $M$-deformations $\mathcal X$ and $\mathcal X'$ over an analytic morphism $f: S\to S'$ is a cartesian diagram
$$
\begin{tikzcd}
\mathcal X \arrow[d]\arrow[r]&\mathcal X'\arrow[d]\\
S \arrow[r,"f"]&S'
\end{tikzcd}
$$
such that $\mathcal X$ and $f^*\mathcal X'$ are isomorphic as $\text{Diff}^0(M)$-bundles over $S$.

The Teichm\"uller stack $\T$ is defined as the stack over the site $\mathfrak S$ such that 
\begin{enumerate}
\item[i)] $\T(S)$ is the groupoid of isomorphism classes of reduced $M$-deformations over $S$.
\item[ii)] $\T(f)$ is the pull-back morphism $f^*$ from  $\mathscr{T}(M)(S')$ to $\mathscr{T}(M)(S)$.
\end{enumerate}

A point $X_0:=(M,J_0)$ is an object of $\mathscr{T}(M)(pt)$ that is a complex structure on $M$ up to biholomorphisms smoothly isotopic to the identity. 

Alternatively, $\T$ can be considered as an analytic version of the quotient $\mathcal I(M)/\text{Diff}^0(M)$. Here, $\mathcal I(M)$ is the set of integrable complex operators on $M$ compatible with its orientation (o.c.), that is
\begin{equation}
\label{I}
\mathcal I(M)=\{J\ :\ TX\longrightarrow TX\mid J^2\equiv -Id,\ J\text{ o.c.},\ \ [T^{1,0}, T^{1,0}]\subset T^{1,0}\}
\end{equation}
for
$$
T^{1,0}=\{v-iJv\mid v\in TX\}.
$$
and $\text{Diff}^0(M)$ is the group of diffeomorphisms of $M$ which are $C^\infty$-isotopic to the identity. 

Two points have to be emphasized here. Firstly, as stack, $\T$ also encodes the isotropy groups of the action. Recall that the isotropy group at $X_0$ is the group
\begin{equation}
\label{Aut1def}
\text{Aut}^1(X_0):=\text{Aut}(X_0)\cap \text{Diff}^0(M).
\end{equation}
which may be different from $\Azero$, the connected component of the identity of the automorphism group $\A$, see \cite{Aut1}. 
Secondly, the action of $\D$ onto $\I$ is not a holomorphic action but an action by biholomorphisms. Indeed $\D$ can be endowed with a complex structure as an open set of the complex Fr\'echet space of $C^\infty$-maps from $M$ to $X_0$, but this complex structure depends on $X_0$, that is depends on the choice of a complex structure on $M$. Taking this into account means putting as complex structure on $\D\times \I$ not a product structure but the structure such that $\D\times \{J\}$ is an open set of the complex Fr\'echet space of $C^\infty$-maps from $M$ to $X_J$. In other words, $\D\times \I$ is endowed with a complex structure as an open set of the complex Fr\'echet space of $C^\infty$-maps from $M\times\I$ to $\X_{\I}$ the tautological family over $\I$.

 Since $\D$ acts on the (infinite-dimensional) analytic space $\I$ preserving its connected components and its irreducible components, we may speak in this way of connected components and irreducible components of $\T$.  Indeed Kuranishi's Theorem tells us that there exist local analytic sections $K_0$ of finite dimension at each point $X_0=(M,J_0)$ of $\I$. Hence, locally, the irreducible components of $\I$ at $J_0$ are those of the finite-dimensional space $K_0$.

In \cite{LMStacks}, a finite-dimensional atlas of (a connected component of)  $\T$ is described under the hypothesis that the dimension of the automorphism group of the complex manifolds encoded in $\T$ is bounded. The rough idea is that the $\text{Diff}^0(M)$-action on $\mathcal I(M)$, though not locally free when the complex structures admit holomorphic vector fields, defines a sort a foliation that we call a TG foliation. A holonomy groupoid can be defined for this sort of foliation and gives the desired atlas.

Such a groupoid is obtained by taking a complete set of local tranversals to the foliation and considering its quotient through the holonomy morphisms. In our situation, the transversal at a point $X_0$ of $\mathscr{T}(M)$ is the Kuranishi stack. It is build from $K_0$, see \cite[\S 2.3]{LMStacks} and Section \ref{Kurstack}. 

\section{The Kuranishi stack}
\label{localstack}
In the first two sections, we review the construction of the Kuranishi family first from classical deformation theory point of view, then from Kuranishi-Douady's point of view.

We then review the construction of the Kuranishi stack(s) introduced in \cite{LMStacks} and finally describe some important properties of them.

It is worth pointing out that the classical point of view (which presents Kuranishi family from a formal/algebraic point of view leaving aside the analytic details of the construction) is not enough for our purposes. This is indeed an infinitesimal point of view and even if it gives complete equations for the Kuranishi space, it fails in describing the properties of the structures close to the base complex structure. Kuranishi-Douady's point of view allows to pass from the infinitesimal point of view to a local one.

\subsection{The Kuranishi family}
\label{Kurfamily}

The Kuranishi family $\pi : \mathscr{K}_0\to K_0$ of $X_0$ is a semi-universal deformation of $X_0$. It comes with a choice of a marking, that is of an isomorphism $i$ between $X_0$ and the fiber $\pi^{-1}(0)$ over the base point $0$ of $K_0$. The semi-universal property means that
\begin{enumerate}[i)]
\item Every marked deformation $\mathscr{X}\to B$ of $X_0$ is locally isomorphic to the pull-back of the Kuranishi family by a pointed holomorphic map defined in a small neighborhood of the base point of $B$ with values in a neighborhood of $0$ in $K_0$.
\item Neither the mapping $f$ nor its germ at the base point are unique; but its differential at the base point is.
\end{enumerate}

Two such semi-universal deformations of $X_0$ are isomorphic up to restriction to a smaller neighborhood of their base points. Hence the germ of deformation $(\mathscr{K}_0,\pi^{-1}(0))\to (K_0,0)$ is unique. This explains why we talk of the Kuranishi family, even if, in many cases, we work with a representent of the germ rather than with the germ itself.

The Zariski tangent space to the Kuranishi space $K_0$ at $0$ identifies naturally with $H^1(X_0,\Theta_0)$. Indeed, $K_0$ is locally isomorphic to an analytic subspace of $H^1(X_0,\Theta_0)$ whose equations coincide at order 2 with the vanishing of the Schouten bracket.

The groups $\A$, $\Azero$ and $\Aun$ act on this tangent space.

However, this infinitesimal action cannot always be integrated in an action of the automorphism groups of $X_0$ onto $K_0$. Still there exists an action of each $1$-parameter subgroup and all these actions can be encoded in an analytic groupoid and thus in a stack. To do this, we need to know more about the complex properties of the structures encoded in a neighborhood of $0$ in $K_0$.

\subsection{Kuranishi-Douady's presentation and automorphisms}
\label{Douady}

Let $V$ be an open neighborhood of $J_0$ in $\I$. Complex structures close to $J_0$ can be encoded as $(0,1)$-forms $\omega$ with values in $T^{1,0}$ which satisfy the equation
\begin{equation}
\label{integ}
\bar\partial \omega+\dfrac{1}{2}[\omega,\omega]=0 
\end{equation}
Choose an hermitian metric and let $\bar\partial^*$ be the $L^2$-adjoint of $\bar\partial$ with respect to this metric. Let $U$ be a neighborhood of $0$ in $(T^{0,1})^*\otimes T^{1,0}$. Set
\begin{equation}
\label{K0}
K_0:=\{\omega\in U\mid \bar\partial \omega+\dfrac{1}{2}[\omega,\omega]=\bar\partial^*\omega=0\}
\end{equation} 
 Let $W$ an open neighborhood of $0$ in the vector space of vector fields $L^2$-orthogonal to the vector space of holomorphic vector fields $H^0(X_0,\Theta)$. In Douady's setting \cite{Douady}, Kuranishi's Theorem states a local isomorphism between $\I$ at $J_0$ and the product of $K_0$ with $W$ such that every plaque $\{pt\}\times L$ is sent through the inverse of this isomorphism into a single local $\D$-orbit. To be more precise, up to restricting $U$, $V$ and $W$, the Kuranishi mapping
\begin{equation}
\label{Kurmap}
(\xi,J)\in W\times K_0\longmapsto J\cdot e(\xi)\in V
\end{equation}
is an isomorphism of infinite-dimensional analytic spaces. As usual, we use the exponential map associated to a riemannian metric on $M$ in order to define the map $e$ which gives a local chart of $\D$ at $Id$. And $\cdot$ denotes the natural right action of $\D$ onto $\I$ given by 
\begin{equation}
\label{actionDiff}
J\cdot f:=df^{-1}\circ J\circ df
\end{equation}
Composing the inverse of \eqref{Kurmap} with the projection onto $K_0$ gives a retraction map $\Xi : V\to K_0$. Let now $f$ be an element of $\A$. There exists some maximal open set $U_f\subset K_0$ such that 
\begin{equation}
\label{Autaction}
\text{Hol}_f\ :\ J\in U_f\subset K_0\longmapsto \Xi(J\cdot f)\in K_0
\end{equation}
is a well defined analytic map. Observe that $\text{Hol}_f$ fixes $J_0$ and that it fixes each leaf of the foliation\footnote{The leaves correspond to the connected components in $K_0$ of the following equivalence relation: $J\equiv J'$ iff both operators belong to the same $\D$-orbit.}  of $K_0$ described in \cite[\S 3]{ENS}.

Of course, in \eqref{Autaction}, we may restrict ourselves to elements of $\Azero$ or $\Aun$.

\subsection{The Kuranishi stacks}
\label{Kurstack}
 The Kuranishi stacks encode the maps \eqref{Autaction} in an analytic groupoid. The first step to do this consists in proving that there is an isomorphism
 \begin{equation}
 \label{Gdecompoglobal}
 (\xi,g)\in W\times \text{\rm Aut}^0(X_0)\longmapsto g\circ e(\xi)\in \mathcal D_0
 \end{equation}
 with values in a neighborhood $\mathcal D_0$ of $\Azero$ in $\D$, see \cite[Lemma 4.2]{LMStacks}. 
 
 Let now $\text{Diff}^0 (M,\mathscr K_0)$ denote the set of $C^\infty$ diffeomorphisms from $M$ to a fiber of the Kuranishi family $\mathscr K_0\to K_0$. This is an infinite-dimensional analytic space\footnote{Strictly speaking, we have to pass to Sobolev $L^2_l$-structures for a big $l$ to have an analytic space, and $\text{Diff}^0 (M,\mathscr K_0)$ is the subset of $C^\infty$ points of this analytic set. In the sequel, we automatically make this slight abuse of terminology, cf. Convention 3.2 in \cite{LMStacks}.}, see \cite{Douady}. Here by $(J,F)\in \text{Diff}^0 (M,\mathscr K_0)$, we mean that we consider $F$ as a diffeomorphism from $M$ to the complex manifold $X_J$.
  
 Given $(J,F)$ an element of $\text{Diff}^0 (M,\mathscr K_0)$, we say it is $(V,\mathcal D_0)$-admissible if there exist a finite sequence $(J_i,F_i)$ (for $0\leq i\leq p$) of $\text{Diff}^0 (M,\mathscr K_0)$ such that
 \begin{enumerate}[i)]
 \item $J_0=J$ and each $J_i$ belongs to $K_0$.
 \item Each $F_i$ belongs to $\mathcal D_0$.
 \item $F=F_0\circ\hdots\circ F_p$
 \item $J_{i+1}=J_i\cdot F_i$
\end{enumerate}
 We set then
 \begin{equation}
 \label{Azero}
 \mathcal A_0=\{(J,F)\in \text{Diff}^0 (M,\mathscr K_0)\mid (J,F)\text{ is }(V,\mathcal D_0)\text{-admissible}\}
 \end{equation}
We also consider the two maps from $\mathcal A_0$ to $K_0$
 \begin{equation}
 \label{st}
 s(J,F)=J\qquad\text{ and }\qquad t(J,F)=J\cdot F
 \end{equation}
 and the composition and inverse maps
 \begin{equation}
 \label{multgroupoid}
 m((J,F),(J\cdot F, F'))=(J,F\circ F'),\qquad i(J,F)=(J\cdot F, F^{-1})
 \end{equation} 
 With these structure maps, the groupoid $\mathcal A_0\rightrightarrows K_0$ is an analytic groupoid \cite[Prop. 4.6]{LMStacks} whose stackification over $\mathfrak{S}$ is called {\slshape the Kuranishi stack} of $X_0$. We denote it by $\mathscr A_0$. Note that it depends indeed of the particular choice of $V$.
 
 As a category, its objects are still reduced $M$-deformations over bases belonging to $\mathfrak S$. However, the allowed complex structures are those encoded in $V$; and the allowed families are those obtained by gluing pull-back families of $\mathscr K_0\to K_0$ with respect to $(V,\mathcal D_0)$-admissible diffeomorphisms. In the same way, morphisms are those induced by $(V,\mathcal D_0)$-admissible diffeomorphisms. Hence, not only the complex fibers of the families have to be isomorphic to those of $\mathscr K_0$, but gluings and morphisms of families are restricted.
 
  Of course, the same construction can be carried out for the automorphism groups $\Aun$, resp. $\A$, with the obvious modifications. In \eqref{Gdecompoglobal}, $\Azero$ is replaced with $\Aun$, resp. $\A$, defining a neighborhood $\mathcal D_1$ of $\Aun$ in $\D$, resp. $\mathcal D$ of $\A$ in $\text{Diff}(M)$. This allows to speak of $(V,\mathcal D_1)$-admissible, resp. $(V,\mathcal D)$-admissible diffeomorphisms. Then, replacing $\mathcal D_0$ with $\mathcal D_1$, resp. $\mathcal D$ in \eqref{Azero} we obtain the analytic groupoid $\mathcal A_1\rightrightarrows K_0$, resp. $\mathcal A\rightrightarrows K_0$. Its stackification over $\mathfrak S$ gives a stack $\mathscr A_1$, resp. $\mathscr{A}$. The previous description of $\mathscr A_0$ as a category applies to $\mathscr{A_1}$, resp. $\mathscr A$ with the obvious changes. We also call them {\slshape Kuranishi stacks}.

\subsection{Universality and the Kuranishi stacks}
\label{universal}
Recall that Kuranishi's Theorem asserts the existence of a semi-universal deformation for any compact complex manifold. This is not however a universal deformation when the dimension of the automorphism group varies in the fibers of the Kuranishi family, i.e. in the setting of section \ref{Kurfamily}, the germ of mapping $f$ is not unique. Replacing the Kuranishi space with the Kuranishi stack allows to recover a universality property. 

To do that, we need to germify the Kuranishi stacks. We replace our base category $\mathfrak{S}$ with the base category $\mathfrak{G}$ of germs of analytic spaces. We turn $\mathfrak G$ into a site by considering the trivial coverings. Hence each object of $\mathfrak G$ has a unique covering and there is no non trivial descent data.

We then germify the groupoids. Starting with $\mathcal A\rightrightarrows K_0$, resp. $\mathcal A_0\rightrightarrows K_0$ and $\mathcal A_1\rightrightarrows K_0$, and using $s$ and $t$ as defined in \eqref{st}, we germify $K_0$ at $0$, $\mathcal{A}$, resp. $\mathcal{A}_0$ and $\mathcal{A}_1$, at the fiber $(s\times t)^{-1}(0)$ and germify consequently all the structure maps. We thus obtain the groupoids $(\mathcal A,(s\times t)^{-1}(0))\rightrightarrows (K_0,0)$, resp. $(\mathcal A_0,(s\times t)^{-1}(0))\rightrightarrows (K_0,0)$ and $(\mathcal A_1,(s\times t)^{-1}(0))\rightrightarrows (K_0,0)$.

Finally, we stackify $(\mathcal A,(s\times t)^{-1}(0))\rightrightarrows (K_0,0)$, resp. $(\mathcal A_0,(s\times t)^{-1}(0))\rightrightarrows (K_0,0)$ and $(\mathcal A_1,(s\times t)^{-1}(0))\rightrightarrows (K_0,0)$, over $\mathfrak{G}$. We denote the corresponding stacks by $(\mathscr{A},0)$, resp. $(\mathscr{A}_0,0)$ and $(\mathscr{A}_1,0)$.

The objects of $(\mathscr{A},0)$ over a germ of analytic space $(S,0)$ are germs of $M$-deformations $p : \mathcal X\to S$ with fiber at the point $0$ of $S$ isomorphic to $X_0$. We denote them by $(\mathcal X,p^{-1}(0))\to (S,0)$. The morphisms over some analytic mapping $f: S\to S'$ are germs of morphisms between $M$ deformations $(\mathcal X,p^{-1}(0))\to (S,0)$ and $(\mathcal X',{p'}^{-1}(0))\to (S',0')$ over $f$. Note that $f(0)=0'$.

\begin{remark}
	\label{rkimportant}
	It is crucial to notice that we deal with germs of {\slshape unmarked} deformations. There is obviously a distinguished point (since we deal with germs), but there is no marking of the distinguished fiber.
\end{remark}

The following theorem shows that $(\mathscr{A},0)$ contains indeed {\slshape all} such germs of $M$-deformations and of morphisms between $M$-deformations. It is folklore although we never saw a paper stating this in a precise way.

\begin{theorem}
	\label{thmuniversal}
	The stack $(\mathscr{A},0)$ is the stack over $\mathfrak G$ whose objects are the germs of $M$-deformations of $X_0$ and whose morphisms are the germs of morphisms between $M$-deformations.
\end{theorem}

\begin{proof}
	Since the site $\mathfrak{G}$ does not contain any non-trivial covering, there is no gluings of families, and the torsors associated to $(\mathscr{A},0)$ are just given by the pull-backs of the germ of Kuranishi family $(\mathscr{K}_0,\pi^{-1}(0))\to (K_0,0)$. Kuranishi's Theorem implies that the natural inclusion of $(\mathscr{A},0)$ in the stack of germs of $M$-deformations over $\mathfrak G$ is essentially surjective.
	
	Morphisms over the identity of some germ $(S,0)$ of analytic space are thus given by morphisms $F$ between two germs of families $(f^*\mathscr{K}_0,\pi^{-1}(0))\to (B,0)$ and $(g^*\mathscr{K}_0,\pi^{-1}(0))\to (B,0)$ for $f$ and $g$ germs of analytic mappings from $(B,0)$ to $(K_0,0)$. Hence $F$ restricted to the central fiber $X_0\simeq \pi^{-1}(0)$ is an automorphism of the central fiber that is an element of $\A$. But $(s\times t)^{-1}(0)$ is isomorphic to $\A$ so such a morphism $F$ is induced by an analytic mapping from $(B,0)$ to $(\mathcal A,(s\times t)^{-1}(0))$ that we still denote by $F$ which satisfies $s\circ F=f$ and $t\circ F=g$. This shows that the natural inclusion of $(\mathscr{A},0)$ in the stack of germs of $M$-deformations over $\mathfrak G$ is fully faithful. 
\end{proof}

This must be thought of as a property of universality. Indeed, the failure of universality in Kuranishi's theorem comes from the existence of automorphisms of the Kuranishi family fixing the central fiber but not all the fibers. Imposing a marking is an artificial and incomplete solution to this problem because this automorphism group is not in general isomorphic to $\A$ which is killed by the marking.

In the same way, we have 
\begin{corollary}
	\label{thmuniversal1}
	The stack $(\mathscr{A}_1,0)$ is the stack over $\mathfrak G$ whose objects are the germs of reduced $M$-deformations and whose morphisms are the germs of morphisms between reduced $M$-deformations.
\end{corollary}
Here $C^\infty$-markings of the $M$-families, that is the choice of a $C^\infty$ diffeomorphism from $M$ to the central fiber, can be used to characterize reduced families. Morphisms are required to induce on $M$ a diffeomorphism isotopic to the identity through the markings.

We also have

\begin{corollary}
	\label{thmuniversal0}
	The stack $(\mathscr{A}_0,0)$ is the stack over $\mathfrak G$ whose objects are the germs of $0$-reduced $M$-deformations and whose morphisms are the germs of morphisms between $0$-reduced $M$-deformations.
\end{corollary}

A $0$-reduced $M$-deformation is just a marked family. We use a different terminology because morphisms are different. A morphism of marked families is required to induce on $X_0$ the identity through the markings, whereas a morphism of $0$-reduced $M$-deformation is required to induce on $X_0$ an element of $\Azero$ through the markings.

\subsection{Kuranishi stack as an orbifold}
\label{Kurorbifold}
We now investigate when the Kuranishi stack(s) is (are) an orbifold. Here by orbifold, we mean the stack given by the global quotient of an analytic space by an holomorphic action of a finite group fixing exactly one point.
We have
\begin{theorem}
	\label{Kurstackorbifold}
	The Kuranishi stack $\mathscr{A}$ is an orbifold if and only if $\A$ is finite.
\end{theorem}
\begin{remark}
	\label{precisestatement}
	To be more precise, the statement "the Kuranishi stack $\mathscr{A}$ is an orbifold" must be understood as: "for a good choice of $ V$ and $K_0$, the Kuranishi stack $\mathscr{A}$ is an orbifold". This should appear clearly in the proof.
\end{remark}

\begin{proof}
	Since the isotropy group of $X_0$ is $\A$, the condition is obviously necessary. So let us assume that $\A$ is finite. We show that we may choose $ V$ and $K_0$ so that the corresponding atlas $\mathcal A\rightrightarrows K_0$ of $\mathscr{A}$ is Morita equivalent to the translation groupoid $\A\times K_0\rightrightarrows K_0$.
	
	We start with an arbitrary atlas $\mathcal A\rightrightarrows K_0$. We assume that the $\A$ version of \eqref{Gdecompoglobal} is valid. Given $f\in \A$, define $\text{Hol}_f$ as in \eqref{Autaction} and set
	\begin{equation}
		\label{sigmaf}
		\sigma_f : J\in U_f\subset K_0\longmapsto (J,f\circ e(\chi_J))\in\mathcal{A}
	\end{equation}
where $\chi$ is an analytic mapping from $U_f\subset K_0$ to $W$ with $\chi(0)=0$ defined by
\begin{equation}
	\label{chichecks}
	\Xi(J\cdot f)=(J\cdot f)\cdot e(\chi(J))
\end{equation}
	Observe that
\begin{equation}
	\label{extension}
	\text{Hol}_f(J)=f\circ e(\chi(J))
\end{equation}
The map $\sigma_f$ is a local analytic section of the source map $s : \mathcal A\to K_0$ defined on $U_f$. It satisfies
\begin{equation}
	\label{tsigma}
	t\circ\sigma=\text{Hol}_f
\end{equation}
The proof of Theorem \ref{Kurstackorbifold} consists in the following two lemmas.
	
	\begin{lemma}
		\label{lemma1}
		For all $f\in\A$, the map $\sigma_f : U_f\subset K_0\to \mathcal A$ is the unique (up to restriction) extension of $f$.
	\end{lemma}

	By extension of $f$, we mean a section $F$ of $s$ defined in a neighborhood of $0$ and such that $F(0)=f$.
	
	\begin{proof}[Proof of Lemma \ref{lemma1}]
	
		The map $\sigma_f$ is obviously an extension of $f$ as desired. 
		
		Let now $G$ be another extension of $f$. Then, for all $J\in K_0$ close to $0$, we have a decomposition
		\begin{equation}
			\label{extensionbis}
			G(J)=f\circ e(\eta(J))
		\end{equation}
		using \eqref{Gdecompoglobal}. Here the factor in $\A$ is constant equal to $f$ since $\A$ is discrete.
		
		The mapping $\eta$ also satisfies \eqref{chichecks}. But since \eqref{Kurmap} is an isomorphism, \eqref{chichecks} is uniquely verified and $\eta=\chi$. Thus $G=\sigma_f$ on a neighborhood of $J_0$ in $K_0$.		
	\end{proof}
	Redefine $K_0$ as the intersection of all $U_f$ for $f$ in the finite group $\A$. Then all extensions $\sigma_f$ are defined on $K_0$,
	Set 
	\begin{equation}
		\label{Ext}
		\mathscr{E}xt=\{\sigma_f : K_0\to \mathscr{A}\mid f\in\A\}
	\end{equation}
	We have
	\begin{lemma}
		\label{lemma2}
		$(\mathscr{E}xt,\circ)$ is a group isomorphic to $\A$.
	\end{lemma}

	\begin{proof}[Proof of Lemma \ref{lemma2}]
		Define 
		\begin{equation}
			\sigma_g\circ \sigma_f\ :\ J\in K_0\longmapsto \sigma_g(J\cdot \sigma_f(J))\circ \sigma_f(J)
		\end{equation} 
		This is an extension of $g\circ f$, and thus by Lemma \ref{lemma1} is equal to $\sigma_{g\circ f}$ on a neighborhood of $J_0$, hence on $K_0$ by analyticity.
		
		Still by Lemma \ref{lemma1}, if $g_1\circ\hdots\circ g_k=Id$, then the same relation holds for the $\sigma_{g_i}$'s.
	\end{proof}
	 The space $K_0$ is invariant by the group $(\mathscr{E}xt,\circ)$. We may thus take as atlas for $\mathscr{A}$ the translation groupoid $\mathscr{E}xt\times K_0\rightrightarrows K_0$, or, equivalently the translation groupoid $\A\times K_0\rightrightarrows K_0$.
\end{proof} 

Replacing $\mathscr{A}$ with $\mathscr{A}_1$, resp. $\mathscr{A}_0$ and $\A$ with $\Aun$, resp. $\Azero$ yields the following immediate corollaries.

\begin{corollary}
		\label{Kur1stackorbifold}
	The Kuranishi stack $\mathscr{A}_1$ is an orbifold if and only if the group $\Aun$ is finite.
\end{corollary}
and
\begin{corollary}
	\label{Kur0stackorbifold}
	The Kuranishi stack $\mathscr{A}_0$ is an orbifold if and only if $\Azero$ is zero.
\end{corollary}

\section{The neighborhood of a point in the Teichm\"uller stack}
\label{localTeich}
A neighborhood of $X_0$ in $\mathscr{T}(M)$ consists of $M$-deformations all of whose fibers are close to $X_0$, that is can be encoded by structures $J$ living in a neighborhood $V$ of $J_0$ in $\mathcal I(M)$. As in \cite{LMStacks}, we shall denote it by $\mathscr{T}(M, V)$. From now on, we assume that $V$ is open, connected and small enough to come equipped with a Kuranishi mapping \eqref{Kurmap}.

\subsection{Atlas}
\label{atlas}
The main difficulty to construct an atlas in \cite{LMStacks} was to describe all the morphisms between the different Kuranishi spaces involved. Here, in the local case, we just need to use one Kuranishi space and it is straightforward to give an atlas for $\mathscr{T}(M, V)$. Just consider
\begin{equation}
	\label{atlasneigh}
	\mathcal T_V:=\{(J,f)\in\text{Diff}^0(M,\mathscr K_0)\mid J\cdot f\in K_0\}
\end{equation}
with structure maps as in \eqref{st} and \eqref{multgroupoid}. This gives an atlas $\mathcal T_V\rightrightarrows K_0$ for $\mathscr{T}(M,V)$. Observe that it is very close to that of the Kuranishi stacks. Indeed the points are exactly the same than those of $\mathscr{A}_1$, but $\mathscr{A}_1$ has less morphisms, hence also less descent data and thus less objects. To understand how to pass from $\mathscr{A}_1$ to $\mathscr{T}(M,V)$, we need to understand and encode the "missing" morphisms. 

\subsection{Target Germification}
\label{proj}
To compare the local Teichm\"uller stack with the Kuranishi stacks, we need to germify $ V$ and consider only complex structures belonging to the germ of some point $J_0$ in $ V$. This process is different from the germification process of section \ref{universal} which was about germifying the base category and thus the base of $M$-deformations. Here we still want to consider $M$-deformations over any analytic bases, but need to germify the set of possible fibers. Hence we need a target germification process, as opposed to the source germification process used in Section \ref{universal}.

To do that, we look at sequences of stacks $\mathscr{T}(M, V_n)$ for $( V_n)$ an inclusion decreasing sequence of neighborhoods of a fixed point $J_0$ with $ V_0= V$. Corresponding to a nesting sequence 
\begin{equation}
	\label{nestingseq}
	\hdots\subset  V_n\subset\hdots\subset  V\subset\mathscr{I}
\end{equation}
we obtain the sequence
\begin{equation}
	\label{nestingsstack}
	\begin{tikzcd}
	\hdots \arrow[r,hook]&\mathscr{T}(M, V_n)\arrow[r,hook]&\hdots\arrow[r,hook]&\mathscr{T}(M, V)
	\end{tikzcd}
\end{equation}
As in the standard case of germ of topological spaces, we may endow the set of stacks $\mathscr{T}(M,V)$ for $V\subset\mathscr{I}$ with the following equivalence relation
\begin{equation}
	\label{eqrel}
	\mathscr{T}(M,V)\sim \mathscr{T}(M,W)\iff V\cap  V'=W\cap  V'
\end{equation}
for some neighborhood $ V'$ of $J_0$ in $\mathscr{I}$.
\vspace{5pt}\\
We call the resulting equivalence class of $\mathscr{T}(M)$ the {\slshape target germification} of $\mathscr{T}(M)$ at $J_0$ and denote it by $(\mathscr{T}(M),J_0)$.
\vspace{5pt}\\
To understand $(\mathscr{T}(M),J_0)$, we need to consider sequences of $M$-deformations over the same base $(\mathscr{X}_n\to B)$ such that $\mathscr{X}_n$ is an object of $\mathscr{T}(M, V_n)$ for some decreasing sequence \eqref{nestingseq}. We identify two such sequences $({\mathscr{X}'}_n\to B)$ and $(\mathscr{X}_n\to B)$ if the families ${\mathscr{X}'}_n\to B$ and $\mathscr{X}_n\to B$ are isomorphic for every large $n$. 
Here are some examples of such sequences
\begin{enumerate}[i)]
	\item Start with a $M$-deformation $\mathscr{X}\to \mathbb{D}$ over the disk with central fiber isomorphic to $X_0$ and thus can be encoded in $J_0$. Then consider the pull-back sequence $(\lambda_n^*\mathscr{X}\to \mathbb{D})$ where $(\lambda_n)$ is a sequence of homotheties with ratio decreasing from $1$ to $0$.
	\item Start with a fiber bundle $E\to B$ with fiber $X_0$ and structural group $\Aun$ and a $M$-deformation $\pi : \mathscr{X}\to B\times\mathbb{D}$ which coincides with the bundle $E$ over $B\times\{0\}$. Then pick up some sequence $(x_n)$ in the disk which converges to $0$. Then consider the sequence of families $(\pi^{-1}(B\times\{x_n\})\to B)$.
\end{enumerate}
Morphisms from $({\mathscr{X}'}_n\to B)$ to $(\mathscr{X}_n\to B)$ are sequences $(f_n)$ with $f_n$ a family morphism over $B$ from ${\mathscr{X}'}_n$ to $\mathscr{X}_n$ for every $n$. Once again, we identify tow such sequences $(f_n)$ and $(g_n)$ if there exists some integer $k$ such that $f_n=g_n$ for $n\geq k$.

\subsection{Sequences of isomorphic structures in the Kuranishi space}
\label{pointstarget}
Let $(f_n)$ be a morphism from $({\mathscr{X}'}_n\to B)$ to $(\mathscr{X}_n\to B)$, two sequences of $(\mathscr{T}(M),J_0)$, as explained above in Section \ref{proj}. Since $\mathcal{T}_{V_n}\rightrightarrows K_0$ is an atlas for $\mathscr{T}(M,V_n)$, each $f_n$ is obtained by gluing a cocycle of morphisms in $\mathcal{T}_{V_n}$ over an open cover of $V_n$. Such morphisms are local morphisms of the Kuranishi family. Hence the existence of a morphism $f_n$ acting non-trivially on the base is subject to the existence of two isomorphic distinct fibers of the Kuranishi family, that is to the existence of two distinct points in $K_0$ encoding the same complex manifold up to isomorphism. In the same way, the existence of sequences of morphisms $(f_n)$ acting non-trivially on the base is subject to the existence of sequences $(x_n)$ and $(y_n)$ of points in $K_0$ such that
\begin{enumerate}[i)]
	\item Both sequences $(x_n)$ and $(y_n)$ converge to the base point of $K_0$.
	\item For every $n$, the fibers of the Kuranishi family above $x_n$ and $y_n$ are isomorphic.
\end{enumerate}
In particular, there exists a sequence $(\phi_n)$ of $\D$ such that
\begin{equation}
	\label{phin}
	\forall n,\qquad x_n\cdot\phi_n=y_n
\end{equation}
Now, we are looking for {\slshape missing} morphisms in the Kuranishi stacks. In other words, we are looking for such sequences $(x_n)$ and $(y_n)$ with the additional property that $(x_n,\phi_n)$ does not belong to $\mathscr{A}_1$, that is $(x_n,\phi_n)$ is not $(V_n,\mathcal D_0)$-admissible.

From \eqref{phin} and the convergence of $(x_n)$ and $(y_n)$, we see that the sequence $(\phi_n)$ may exhibit three different types of behaviour.
\begin{enumerate}[i)]
	\item It converges uniformly on compact sets to a diffeomorphism $g$. This $g$ fixes the base point of $K_0$, hence is an automorphism of $X_0$ and $(\phi_n)$ is already encoded in the Kuranishi stacks.
	\item  It does not converge uniformly but the associated sequence of graphs converges in the cycle space $\mathscr{C}$.
	\item Neither the sequence of diffeomorphisms nor the sequence of graphs converge.
\end{enumerate}
Here $\mathscr{C}$ denotes the Barlet space of (relative) $n$-cycles of $\mathscr{K}^{red}_0\times\mathscr{K}^{red}_0$ for $\mathscr{K}^{red}_0$ the reduction of the Kuranishi family. Let also $\mathscr{C}_0$ be the Barlet space of $n$-cycles of $X_0\times X_0$. 

Each $\phi_n$ is an isomorphism between the fiber $\pi^{-1}(x_n)$ and the fiber $\pi^{-1}(y_n)$ of the Kuranishi family, hence defines an element $\gamma_n$ of $\mathscr{C}$. When these cycles converge, the limit belongs to $\mathscr{C}_0$. So we may reformulate the three previous cases as follows.
\begin{enumerate}[i)]
	\item The cycles $\gamma_n$ converge to the graph of an automorphism of $X_0$.
	\item The cycles $\gamma_n$ converge to a cycle in $\mathscr{C}_0$ which does not belong to the irreducible component(s) formed by $\Aun$.
	\item The cycles $\gamma_n$ do not converge.
\end{enumerate}

\section{Finiteness properties in the K\"ahler setting}
\label{Lieb}
In this section, we recall and apply a basic result on cycle spaces in the K\"ahler case, which is due to Lieberman \cite{Lieberman}. We state the relative version, which is adapted to our purposes.

\begin{proposition}
\label{Liebprop}
Let $\pi_i : \mathcal X_i\to B_i$ be smooth morphisms with compact K\"ahler fibers over reduced analytic spaces $B_i$ for $i=0,1$. Let $E\subset B_0\times B_1$ be a subset and let $\mathcal Z\to E$ be a family of relative cycles of $\mathcal X_0\times \mathcal X_1$. Assume that the projection of $E$ is included in a compact of $B_0\times B_1$. Assume also that any cycle $Z$ is the graph of a biholomorphism from some fiber $(X_0)_t$ onto some fiber $(X_1)_{t'}$. Then,
\begin{enumerate}
\item[i)] $\mathcal Z$ only meets a finite number of irreducible components of the cycle space of $\mathcal X_0\times \mathcal X_1$.
\item[ii)] Let $\mathcal C$ be such a component. Then $\mathcal C$ contains a Zariski open subset $\mathcal C_0$ all of whose members are graphs of a biholomorphism between a fiber of $\mathcal X_0$ and a fiber of $\mathcal X_1$.
\end{enumerate}
\end{proposition}

\begin{proof}
i) Let $(\omega^i_t)_{t\in B_i}$ be a continuous family of K\"ahler forms on the $\pi_i$-fibers ($i=0,1$). Let $M$ be the smooth model of $X_0$ and let $(J^i_t)_{t\in B_i}$ be a continuous family of integrable almost complex operators on $M$ such that $(X_i)_t=(M,J^i_t)$. For every $e\in E$, call $f_e: M\to M$ the biholomorphism from some fiber $(X_0)_t$ onto some fiber $(X_1)_{t'}$ corresponding to the cycle $Z_e$. We compute the volume of these cycles using the $\omega_t$. We have
$$
\text{Vol}(Z_e)=\int_M \left (\omega^0_t+f_e^*\omega^1_{t'}\right )^n=\int_M\left (\omega^0_t+\omega^1_{t'}\right )^n
$$
since $f_e$ is isotopic to the identity hence $f_e^*\omega^1_{t'}$ and $\omega^1_{t'}$ differs from an exact form. Since the projection of $E$ is included in a compact of $B_0\times B_1$, we obtain that the volume of the $Z_e$ is uniformly bounded. It follows from \cite[Theorem 1]{Lieberman} that $\mathcal Z$ has compact closure in the cycle space of $\mathcal X_0\times\mathcal X_1$. Hence $\mathcal Z$ only meets a finite number of irreducible components of this cycle space.\vspace{3pt}\\
ii) Consider the family of cycles $\tilde{\mathcal C}\subset \mathcal X_0\times\mathcal X_1\to \mathcal C$. Since this map is proper and surjective, it is smooth on a Zariski open subset. Since some fibers are non singular, the generic fiber is non singular. The cycles above $E$ are submanifolds of some $(X_0)_t\times (X_1)_{t'}$ with projections $pr_i$ being bijective onto both factors. Hence, on a Zariski open subset of $\mathcal C$, every cycle enjoys such properties. So is the graph of a biholomorphism between a fiber of $\mathcal X_0$ and a fiber of $\mathcal X_1$.

\end{proof}
Setting $X_0=(M,J_0)$, considering the orbit $\mathcal O$ of $J_0$ in $\mathcal I(M)$ and viewing $K_0$ as a local transverse section, we obtain a first interesting Corollary. 

\begin{corollary}
	\label{interlemma}
	If $K_0$ is small enough then $K_0$  intersects $\mathcal O$ only at $J_0$. 
\end{corollary}

\begin{proof}
	We assume that $K_0$ is small enough so that it only contains K\"ahler points. We also assume that $K_0$ is reduced, replacing it with its reduction if necessary. Take $\X_1=\mathscr{K}_0$ and $\X_0=X_0$ seen as a family over the point $\{J_0\}$. Let $E'$ be the subset of $K_0$ corresponding to complex structures $J$ in the same $\D$-orbit that $J_0$. Now, $E'$ is a subset of $K_0$ which does not contain any continuous path by Fischer-Grauert Theorem, see \cite[Lemma 5]{ENS} for the convenient geometric reformulation. So it contains at most a countable number of points. Since we are only interested in what happens close to $J_0$, we may replace $E'$ with its intersection with a compact neighborhood of $J_0$ in $K_0$. Then for each $J\in E'$, choose some element $f_J$ of $\D$ mapping $J_0$ onto it. Set $E=\{J_0\}\times E'$ and let $\mathcal Z$ be the cycles corresponding to the graphs of the $f_J$. Apply Proposition \ref{Liebprop}. We conclude that $\mathcal Z$ meets a finite number of irreducible components of the cycle space of $X_0\times\mathscr{K}_0$, say $\mathcal C_1$,..., $\mathcal C_p$.
	
	 Still by Proposition \ref{Liebprop}, it follows that a Zariski open subset of each $\mathcal C_i$ only contains graph of biholomorphisms between $X_0$ and some $X_J$ with $J\in E'$. Hence each of these components only contains cycles in a fixed product $X_0\times X_J$ and $E'$ is a finite subset. Reducing $K_0$ if necessary, we may assume that $E'$ is just $\{J_0\}$ as wanted.
\end{proof}

Let $\Gamma$ be the union  of the irreducible components of $\mathscr{C}$ containing a sequence \eqref{phin}. Let $\Gamma_0$ be the intersection of $\Gamma$ with $\mathscr{C}_0$. Observe that every cycle in $\Gamma_0$ projects onto each factor of $X_0\times X_0$ through a map of degree one. In the same way, every cycle in $\Gamma$ projects onto each factor of some $X_t\times X_{t'}$ through a map of degree one. We may state our second Corollary

\begin{corollary}
	\label{2ndcor}
	Assume $X_0$ is K\"ahler. Then,
	\begin{enumerate}[\rm i)]
		\item Each irreducible component of $\Gamma$ contains at least one cycle of $X_0\times X_0$.
		\item $\Gamma$ has a finite number of components.
		\item Take an irreducible component $\mathcal{C}$ of $\Gamma$. Every connected component of the intersection $\Gamma_0\cap \mathcal C$ is either the closure of connected components of $\Aun$ or a component all of whose members are singular cycles.
	\end{enumerate}
\end{corollary}

\begin{definition}
	\label{defex}
	An exceptional component of $\Gamma_0$ is a connected component of some $\Gamma_0\cap \mathcal C$, all of whose members are singular cycles. In this case, we also say that $\mathcal C$ is exceptional above $J_0$.
\end{definition}
Exceptional components may be isolated singular cycles or components of positive dimension with a reducible generic member. 

Notice that the connected components of an intersection $\Gamma_0\cap \mathcal C$ may not be connected components of $\Gamma_0$. Two distinct irreducible components of $\Gamma$ 
may intersect in $\Gamma_0$, for example in the singular part of the closure of a connected component of $\Aun$.

\begin{proof}
	As above in the proof of Corollary \ref{interlemma}, assume that $K_0$ is reduced and all the fibers of the Kuranishi family are Kähler. We may thus apply Proposition \ref{Liebprop} to $\mathcal X_0=\mathcal X_1=\mathscr{K}_0$ and to $\Gamma$. This proves that $\Gamma$ has a finite number of components. It also proves that every sequence of $\Gamma$ converges up to passing to a subsequence, and so case iii) of the list at the end of Section \ref{localTeich} is not possible. 

	 Moreover, if a component of some $\Gamma_0\cap \mathcal C$ contains a (irreducible) manifold, this has to be the graph of an automorphism, since it projects with degree one onto each factor of $X_0\times X_0$. Thus it must contain all the extensions of this automorphisms, hence a connected component of $\Aun$. By Proposition \ref{Liebprop} applied to $\mathcal X_0=\mathcal X_1=X_0$ and to $\Gamma_0$, this is a component of the closure of $\Aun$. Now such a component may connect several components of $\Aun$.
\end{proof}

More precisely, let $\text{Irr }\Gamma$ be the set of irreducible components of $\Gamma$; that is, each point of $\text{Irr }\Gamma$, encodes an irreducible component of $\Gamma$.

\begin{corollary}
	\label{3rdcor}
	Assume $X_0$ is K\"ahler. Then,
	\begin{enumerate}[\rm i)]
		\item The number of irreducible components of the reduction of $\mathcal T_V$ is finite.
		\item If $V$ is a sufficiently small neighborhood of $J_0$ in $\I$, then there exists a natural bijection between the set of irreducible components of the reduction of $\mathcal T_V$ and $\text{Irr }\Gamma$.
		\item If $V$ is a sufficiently small neighborhood of $J_0$ in $\I$, and $V'\subset V$ contains also $J_0$, then the natural inclusion of $\mathcal T_{V'}$ in $\mathcal T_V$ is a bijection between the corresponding sets of irreducible components.
	\end{enumerate}
\end{corollary}

\begin{proof}
	Assume $\mathcal T_V$ is reduced. Every irreducible component of $\mathcal T_V$ injects in an irreducible component of $\mathscr C$. Just send $(J,f)$ to its graph as a cycle of $\mathscr K_0\times\mathscr K_0$. Since $\mathscr C$ has only a finite number of components by K\"ahlerianity, this proves i). Moreover, by Proposition \ref{Liebprop}, a component of $\mathcal T_V$ forms a Zariski open subset of the corresponding component of $\mathscr C$. Then, using Corollary \ref{2ndcor}, for a sufficiently small neighborhood $V$ of $J_0$, the point $J_0$ is adherent to the $s$-image of every such component. Hence they contain a sequence \eqref{phin}. So the irreducible components of  $\mathcal T_V$ are in fact in 1:1 correspondence with the irreducible components of $\Gamma$, proving ii). Then iii) follows from ii).
\end{proof}
As a consequence of Corollary \ref{3rdcor}, we do not need to consider the full target germification in the K\"ahler case. It is enough to look at $\mathscr{T}(M,V)$ for a fixed small enough $V$ since restricting to smaller neighborhoods of $0$ will not change the number of components of $\mathcal{T}_V$.

In other words, there is no wandering sequence \eqref{phin} with each $\phi_n$ belonging to a different component of $\mathcal T_V$.

\section{Local Structure of the Teichm\"uller Stack in the K\"ahler setting}
\label{LocalTeichI}
We want to analyse the structure of the analytic space $\mathcal T_V$ defined in \eqref{atlasneigh} and compare it with $\mathscr{A}_1$ in the K\"ahler setting. 

We already observed in Section \ref{localTeich} that there is a natural inclusion of groupoids of $\mathcal{A}_1$ into $\mathcal T_V$. It comes from the fact that $\mathcal T_V$ encodes every morphism between fibers of the Kuranishi family, whereas $\mathcal{A}_1$ encodes {\slshape some} morphisms between fibers of the Kuranishi family. This inclusion is just the description at the level of atlases of the natural inclusion of $\mathscr{A}_1$ into $\mathscr{T}(M,V)$: $\mathscr{A}_1$-objects, resp. $\mathscr{A}_1$-morphisms, inject in $\mathscr{T}(M,V)$-objects, resp. $\mathscr{T}(M,V)$-morphisms. So our final goal here is to give the structure of this inclusion.

There exists also a natural inclusion of $\mathscr{A}_0$ into $\mathscr{T}(M,V)$. We first relate the morphisms encoded in $\mathcal T_V$ to those encoded in $\mathcal{A}_0$.

\begin{lemma}
	\label{lemma1main}
	Let $(J,f)\in \mathcal T_V$ and let $(J,g)\in \mathcal T_V$. Then these two elements belong to the same connected component of $\mathcal T_V$ if and only if $(J\cdot f,f^{-1}\circ g)$ belongs to $\mathcal{A}_0$.
\end{lemma}

\begin{proof}
	Assume $(J\cdot f,f^{-1}\circ g)$ belongs to $\mathcal{A}_0$, that is $(J\cdot f,f^{-1}\circ g)$ is $(V, \mathcal{D}_0)$-admissible. Then we may decompose it as
	\begin{equation*}
		(J_1,f^{-1}\circ g)=(J_1,h_1)\circ (J_2,h_2)\circ\hdots\circ (J_k,h_k)
	\end{equation*}
where $J_1:=J\cdot f$. Now each $h_i\in \mathcal{D}_0$ can be decomposed as $k_i\circ e(\chi_i)$ through \eqref{Gdecompoglobal} and is thus isotopic to the identity inside $\mathcal{D}_0$. Hence $(J_1,Id)$ and $(J_1,h_1)$
stay in the same connected component of $\mathcal{T}_V$, and so do $(J,f)$ and $(J,f\circ h_1)$.

Conversely, let $(J,f)$ and $(J,g)$ belong to the same connected component of $\mathcal{T}_V$. Then, there exists an isotopy $(J_t,f_t)$ joining these two points in $\mathcal{T}_V$. But then we may find $0<t_1<\hdots <t_k<1$ such that
\begin{equation*}
	(J_1,h_1):=(J\cdot f,f^{-1}\circ f_{t_1})
\end{equation*}
satisfies $h_1\in U$, as well as
\begin{equation*}
	(J_2,h_2):=(J_1\cdot h_1,f_{t_1}^{-1}\circ f_{t_2})
\end{equation*}
satisfies $h_2\in U$ and so on. 
\end{proof}

We are now in position to state and prove our first main result.

\begin{theorem}
	\label{finitethm}
	The natural inclusion of $\mathscr{A}_0$ into $\mathscr{T}(M,V)$, resp. of $\mathscr{A}_1$ into $\mathscr{T}(M,V)$, is a finite \'etale morphism of analytic stacks.
\end{theorem}

By finite \'etale morphism of analytic stacks, we mean that, given any $B\in \mathfrak{S}$ and any morphism $u$ from $B$ to $\mathscr{T}(M,V)$, the fiber product
\begin{equation}
	\label{fb}
	\begin{tikzcd}
		B\times_{u}\mathscr A_0 \arrow[d,"f_1"']\arrow[r,"f_2"]\arrow[dr,phantom,"\scriptstyle{\square}",very near start, shift right=0.5ex]&\mathscr{A}_0\arrow[d,"\text{inclusion}"]\\
		B\arrow[r,"u"]&\mathscr{T}(M,V)
	\end{tikzcd}
,\quad\text{resp.}
\begin{tikzcd}
	B\times_{u}\mathscr A_1 \arrow[d,"f_1"']\arrow[r,"f_2"]\arrow[dr,phantom,"\scriptstyle{\square}",very near start, shift right=0.5ex]&\mathscr{A}_1\arrow[d,"\text{inclusion}"]\\
	B\arrow[r,"u"]&\mathscr{T}(M,V)
\end{tikzcd}
\end{equation}
satisfies
\begin{enumerate}[i)]
	\item $B\times_{u}\mathscr A_0$, resp. $B\times_{u}\mathscr A_1$, is a $\C$-analytic space.
	\item The morphism $f_1$ is a finite \'etale morphism between  $\C$-analytic spaces.
\end{enumerate}

\begin{remark}
	\label{important}
	We emphasize that in the definition of analytic stack used in \cite{LMStacks}, we did not impose that the diagonal is representable, see the discussion in \S 2.4 of \cite{LMStacks}. As a consequence, point (i) above does not follow from the fact that $\mathscr{T}(M,V)$ is an analytic stack but shall be proved by hands.
\end{remark}

\begin{proof}
	By Yoneda's lemma, a morphism $u : B\to\mathscr{T}(M,V)$ corresponds to a family $\mathcal X\to B$. The fiber product $B\times_{u}\mathscr A_0$ encodes the isomorphisms of $\mathscr{T}(M,V)$ 
	\begin{equation}
		\begin{tikzcd}[column sep=small]
			\label{objects}
			\mathcal X \arrow[rr,"\alpha"] \arrow[rd]&&\mathcal X'\arrow[ld]\\
			&B&
		\end{tikzcd}
	\end{equation}
between families $\mathcal X\to \mathcal X'$ over $B$ 
modulo isomorphisms $\beta$ over $B$
 \begin{equation}
 	\begin{tikzcd}[row sep=small]
 		\label{morphisms}
 		&\mathcal X'\arrow[dd, "\beta"]\\
 		\mathcal X \arrow[ur,"\alpha"] \arrow[rd, "\alpha'"']&\\
 		&\mathcal X''
 	\end{tikzcd}
 \end{equation}
belonging to $\mathscr{A}_0$.
	
	Decompose $B$ as a union of connected open sets $B_1\cup\hdots\cup B_k$ in such a way that the family $\mathcal X$, resp. $\mathcal X'$, 
	is locally isomorphic above $B_i$ to $u_i^*\mathscr{K}_0$ for some $u_i : B_i\to K_0$, resp. to $(u'_i)^*\mathscr{K}_0$ for some $u'_i : B_i\to K_0$. These local models are glued through a cocyle $u_{ij}:B_i\cap B_j\to \mathcal T_V$, resp. $u'_{ij}:B_i\cap B_j\to \mathcal A_0$, 
	satisfying $s(u_{ij})=u_i$ and $t(u_{ij})=u_j$, resp. $s(u'_{ij})=u'_i$ and $t(u'_{ij})=u'_j$, to obtain a family isomorphic to $\mathcal X$, resp. $\mathcal X'$.
	
	In these models, up to passing to a finer covering, an isomorphism \eqref{objects} corresponds to a collection $F_i:B_i\to \mathcal T_V$ fulfilling
	\begin{enumerate}[i)]
		\item $s\circ F_i=u_i$ and $t\circ F_i=u'_i$
		\item $F_i\circ u_{ij}=u'_{ij}\circ F_j$
	\end{enumerate} 
Then $\mathcal X''$ corresponds to a cocycle $u_i'' : B_i\to K_0$ and $\alpha'$ to a collection $F'_i:B_i\to\mathcal T_V$ satisfying similar relations.\\
Denoting by $\alpha_i$ the isomorphism between $u_i^*\mathscr{K}_0$ and $(u'_i)^*\mathscr{K}_0$, and setting
\begin{equation}
	\label{Fi}
	F_i(b)=(u_i(b),F_{i,b})\qquad\text{ and }\qquad F'_i(b)=(u_i(b),F'_{i,b})
\end{equation}
we have
\begin{equation}
	\label{alpha}
	\alpha_i(b,v)=(u_i(b),F_{i,b}(v))\qquad\text{ and }\qquad \alpha'_i(b,v)=(u_i(b),F'_{i,b}(v))
\end{equation}
and
\begin{equation}
	\label{uupus}
	u_i(b)\cdot F_{i,b}=u'_i(b)\qquad\text{ and }\qquad 
	u_i(b)\cdot F'_{i,b}=u''_i(b)
\end{equation}

Let $\beta$ be $\alpha'\circ\alpha^{-1}$. This is a morphism of $\mathscr{T}(M,V)$ which is given in our localisation by 
\begin{equation}
	\label{beta}
	\beta_i(b,v)=(u'_i(b),G_{i,b}(v):=F_{i,b}^{-1}\circ F'_{i,b}(v))
\end{equation}
We want to know when $\beta$ is a morphism of $\mathcal{A}_0$. 

Since the $B_i$ are connected, the image of each map $F_i$, $F'_i$ and $G_i$, is included in a single connected component of $\mathcal T_V$. By Lemma \ref{lemma1main}, $F_i$ and $F'_i$ land in the same connected component of $\mathcal T_V$ if and only if $G_i$ lands in $\mathcal{A}_0$. 

Choose a point $b_i$ in each $B_i$. Then $\alpha$ and $\alpha'$ are equivalent through \eqref{morphisms} if and only if $(b_i,F_i(b_i))$ and $(b_i,F'_i(b_i))$ belong to the same connected component of $\mathcal T_V$ for all i. 

Now, assume that $(b_1,F_1(b_1))$ and $(b_1,F'_1(b_1))$ belong to the same connected component of $\mathcal T_V$. Given $i\not = 1$ and taking $c\in B_1\cap B_i$, it follows from the compatibility relations that 
\begin{equation}
	\label{compatibility}
	F_i(c)=u'_{i1}\circ F_1\circ u_{1i}(c)\qquad\text{ and }\qquad F'_i(c)=u''_{i1}\circ F'_1\circ u_{1i}(c)
\end{equation} 
But $u'_{i1}$ and $u''_{i1}$ are mappings with values in $\mathcal A_0$, hence $(c,F_i(c))$ and $(c,F'_i(c))$ belong to the same connected component of $\mathcal T_V$. And so do $(b_i,F_i(b_i))$ and $(b_i,F'_i(b_i))$.

Therefore, the fiber product $B\times_u\mathscr{A}_0$ identifies with the disjoint union of $g$ copies of $B$, where $g$ is at most the number of connected components of $\mathcal T_V$, through the map
\begin{equation}
	\label{RepFP}
	(b,(u(b),F_b))\in (B\times_{u}\mathscr A_0)_{\{b\}}\longmapsto (b,\sharp (u(b),F_b))\in B\times\sharp\mathcal T_V
\end{equation}
Here $(B\times_{u}\mathscr A_0)_{\{b\}}$ denotes the set of objects above $\{b\}$, the set $\sharp\mathcal T_V$ is the set of connected components of $\mathcal T_V$ and the $\sharp$ application maps an element of $\mathcal T_V$ to the connected component of $\mathcal T_V$ which contains it.

Finally $f_1$ can be rewritten as the natural projection map
\begin{equation}
	\label{finite1}
	B\times \sharp_g\mathcal T_V\longrightarrow B
\end{equation}
for $\sharp_g\mathcal T_V$ the connected components of $\mathcal T_V$ that can be attained through \eqref{RepFP}. This proves that the inclusion of $\mathscr{A}_0$ in $\mathscr{T}(M,V)$ is a finite \'etale morphism.

Recall that any element $f$ of $\Aun$ admits a natural extension $\sigma_f$ - see \eqref{sigmaf} - in $\mathcal T_V$ whose $s$- and $t$-projections cover a neighborhood of $J_0$ in $K_0$. Recall also that the $s$-projection of any component is at least adherent to $J_0$. Hence, given an element $(J,g)$ of some connected component $\mathcal C$ of $\mathcal T_V$ and an element $f$ of some connected component $A'$ of $\Aun$, we may compose $(J,g)$ with $\sigma_f(J\cdot g)$ as soon as $J$ is sufficiently close to $J_0$. Moreover this composition lands in some component $\mathcal C'$ of $\mathcal T_V$ which depends only on $\mathcal C$ and on $A'$. Hence there is an action of $\sharp\Aun$ onto $\sharp\mathcal T_V$. \\
One eventually finds that the fiber product $B\times_u\mathscr{A}_1$ identifies with the disjoint union of a finite number of copies of $B$ and $f_1$ with the natural projection map
\begin{equation}
	\label{finite2}
	B\times \sharp_g\mathcal T_V/(\sharp\Aun\cap\sharp_g\mathcal T_V)\longrightarrow B
\end{equation}
\end{proof}
So, going back to the setting "Kur/Teich", we obtain that, in the K\"ahler case, there is a finite \'etale morphism from Kur to Teich. Our next step is to characterize the Kur=Teich case. 

The generic fiber of this morphism is $\sharp\mathcal T_V/\sharp\Aun$, hence we obtain Kur=Teich if and only these two finite groups have same cardinal. 
Now, this occurs if and only if, for some $\mathcal C$ in $\text{Irr }\Gamma$, the intersection $\Gamma_0\cap\mathcal C$ contains a component with only singular cycles. 

This motivates the following definition.

\begin{definition}
	\label{defexceptional}
	A point $X_0$ of $\mathscr{T}(M)$ is called {\slshape exceptional} if $\Gamma_0$ has at least one exceptional component in the sense of Definition \ref{defex}.
\end{definition}

By the mere definition, $X_0$ satisfies Kur=Teich if and only if it is not an exceptional point of the Teichm\"uller stack. The atlas \eqref{atlasneigh} being an atlas of a neighborhood of $X_0$ in the Teichm\"uller stack, it contains all the information we need to decide which points close to $X_0$ are exceptional.  

Indeed, pick a point $X_J$ in $K_0$. Assume it is exceptional. If $V$ is sufficiently small, every irreducible component of $\mathscr{C}$ contains a cycle of $X_0\times X_0$, hence the set $\Gamma_J$ of cycles of $X_J\times X_J$ which are limits of cycles of $\mathscr{C}$ is included in $\Gamma$. As a consequence, $J$ is exceptional if and only if there are only singular cycles in a component of $\Gamma$ above $J$. Let $S$ be the analytic set of singular cycles of $\mathscr{C}$. Let $p$ denote the projection of $\Gamma$ to $K_0\times K_0$. This is a proper map, since we are in the K\"ahler setting. The set of exceptional points is equal to
\begin{equation}
	\label{EC}
		E=\bigcup_{\mathcal{C}\in\text{Irr }\Gamma}\{J\in K_0\mid p^{-1}(J,J)\cap\mathcal{C}\not =\emptyset
		\text{ and }(p^{-1}(J,J)\cap\mathcal{C})_0\subset S\}
\end{equation}
where $_0$ means that some connected component of $p^{-1}(J,J)\cap\mathcal{C}$ is included in $S$. Let $E^c$ be the closure of \eqref{EC}.

We claim that $E^c$ is an analytic subspace of $K_0$. To see that, we first embed $E$ in $K_0\times K_0$ through the diagonal embedding of $K_0$. We still call $E$ the image. Let $\mathcal C$ be a component of $\Gamma$ and let $(J_0,J_0)$ be a point of $E$ belonging to the $\mathcal C$-component of \eqref{EC}. Let $W$ be an open neighborhood of $(J_0,J_0)$ in $K_0\times K_0$. We decompose $p^{-1}(J,J)\cap\mathcal{C}$ into connected components $(p^{-1}(J,J)\cap\mathcal{C})_i$ for $i$ between $0$ and $k$. Let $A_i$ be open neighborhoods of $(p^{-1}(J,J)\cap\mathcal{C})_i$ in $\mathcal C$ such that the union of all $A_i$ is exactly $p^{-1}(W)$. We assume that $(J_0,J_0)$ is generic in the sense that the intersection $A_i\cap p^{-1}(J,J)$ is still connected for $(J,J)\in W\cap E$. Notice that the restriction of $p$ to some $A_i$ is still a proper map. We may decompose $E\cap W$ as follows.
\begin{equation}
	\label{constructibleset}
	\begin{aligned}
		E\cap W&=
		\bigcup_{0\leq i\leq k}\{
		(J,J)\in W\mid p^{-1}(J,J)\cap A_i\not =\emptyset\\
		&\qquad\qquad\qquad\text{ and }p^{-1}(J,J)\cap A_i\subset S
		\}\\
		&=\bigcup_{0\leq i\leq k} \Delta\cap (W\setminus p(A_i\setminus S))
	\end{aligned}
\end{equation}
where $\Delta$ is the diagonal of $K_0\times K_0$.
This is a constructible set of $\Delta\cap W$. Hence its closure is an analytic set of $\Delta\cap W$. But its closure is just $E^c\cap W$. So we obtain a chart of analytic subspace for $E^c$ at every generic point. Assume now that $(J_0,J_0)$ is not generic. Then we may perform exactly the same construction but the resulting constructible set \eqref{constructibleset} may forget some exceptional points. Now, let $a$ be a positive integer. The set of points $(J,J)$ of $W$ where $p^{-1}(J,J)\cap A_i$ has exactly $a$ connected components is constructible, see \cite{Stacksproject}, Lemma 37.26.6 in an algebraic context. So is its pull-back by $p$ in $A_i$. Looking at its connected components, we may decompose $A_i$ into a finite set of constructible sets $A_{ij}$ on which $p$ has connected fibers. Then we obtain the correct decomposition
\begin{equation}
	\label{constructibleset2}
		E\cap W=\bigcup_{i,j} \Delta\cap (W\setminus p(A_{ij}\setminus S))
\end{equation}
making of $E\cap W$ a constructible set and of $E^c\cap W$ an analytic one.\\
We claim that $E^c$ is a {\slshape strict} analytic subspace of $K_0$\footnote{As above in Section \ref{Lieb}, we replace $K_0$ with its reduction if necessary, so strict means that $E^c$ is not the whole $K_0$.}. Assume the contrary. Then every cycle of a component $\mathcal C$ above a Zariski open set of the diagonal in $K_0\times K_0$ is singular. As we already argued several times, each component $\mathcal C$ contains a Zariski open subset of graphs of biholomorphisms between fibers of the Kuranishi family. Hence, $p(\mathcal C)$ is an analytic set of $K_0\times K_0$ strictly containing the diagonal; and a Zariski open subset of it encodes graphs of biholomorphisms. We may thus find for a well chosen $J$ close to $0$ in $K_0$, a path of biholomorphisms between $X_J$ and some $X_{J'_t}$ with $J'_t$ distinct from $J$. Hence $K_0$ has a non trivial foliated structure in the sense of \cite{ENS}. But this implies that the dimension of $\Azero$ jumps at $0$, that is is not constant in a neighborhood of $0$ in $K_0$. Since $\mathscr{K}_0\to K_0$ is complete at every point $J$ of $K_0$, denoting its Kuranishi space $K_J$, then the set of exceptional points in $K_J$ is also the full $K_J$. Hence the same argument tells that the dimension of the automorphism group also jumps at $X_J$. But it cannot jump at every point of $K_0$. Contradiction. The set $E^c$ is a strict analytic subspace of $K_0$.

So we may define a strict analytic substack of $\mathscr T(M,V)$ as the stackification of the full subgroupoid of $\mathcal T_V\rightrightarrows K_0$ above $E^c\subset K_0$. Since the notion of exceptional point is an intrinsic notion, this substack is just a neighborhood of $X_0$ of an analytic substack of $\mathscr{T}(M,V_K)$. Here $V_K$ is the open\footnote{By a classical result of Kodaira, K\"ahlerianity is a stable property through small deformations.} set of K\"ahler points of $\I$. So we have proved our second main Theorem
\begin{theorem}
	\label{2ndmainthm}
	The closure of exceptional points of the Teichm\"uller stack $\mathscr{T}(M,V_K)$ of K\"ahlerian structures form a strict analytic substack $\mathscr{E}(M)$ of $\mathscr{T}(M, V_K)$. 
\end{theorem}
and its immediate Corollary
\begin{corollary}
	\label{Cor2ndmain}
	A compact complex K\"ahler manifold $X$ satisfies Kur=Teich as well as the structures sufficiently close to it if and only if it belongs to $\mathscr{T}(M, V_K)\setminus \mathscr{E}(M)$.
\end{corollary}
	 Hence K\"ahler points such that Kur=Teich as well as the structures sufficiently close to it fill a Zariski open substack of the Teichm\"uller stack.
	 
\begin{remark}
	\label{Isubspace}
	The closure of exceptional points also form a strict analytic subspace of $V_K$.
\end{remark}

\begin{remark}
	\label{topgenericity}
	In particular, the set of points satisfying Kur=Teich is dense in the K\"ahlerian Teichm\"uller space $V_K/\D$ and contains an open set. However, due to the non-Hausdorff topology this space may have, this may be a misleading statement. For example, if $M$ is $\mathbb S^2\times\mathbb S^2$, then the (K\"ahlerian) Teichm\"uller space of $M$, as a set, is $\mathbb Z$, a point $a\in\mathbb Z$ encoding the Hirzebruch surface $\mathbb F_{2a}$\footnote{The surfaces $\mathbb F_{2a}$ and $\mathbb F_{-2a}$ are isomorphic, but not through a biholomorphism isotopic to the identity.}. Now, the topology to put on $\mathbb Z$ has for (non trivial) open sets $\{0\}$, $\{0,1\}$, $\{-1,0\}$ and so on, cf. \cite{LMStacks}, Examples 5.14 and 12.6. Hence $0$ is an open and dense subset of the Teichm\"uller space.
\end{remark}

\begin{remark}
	\label{exexceptional}
	If the intersection of an exceptional component of $X_0$ and $\Gamma_J$ is non-empty but contains regular cycles, then the corresponding morphisms 
	\begin{enumerate}[i)]
		\item either form a component of $\text{Aut}^1(X_J)$ which is not induced by $\Aun$.
		\item or send $J$ to points that are distinct from $J$ and not adherent to it. Hence, when restricting to a sufficiently small neighborhood $W$ of $X_J$, this component disappears from $\mathscr{T}(M,W)$.
	\end{enumerate}
\end{remark}

Finally, we deal with the orbifold case.

\begin{theorem}
	\label{mainorbifoldthm}
	Assume $X_0$ K\"ahler. Then, the following statements are equivalent
	\begin{enumerate}[\rm i)]
		\item $\mathscr{T}(M,V)$ is an orbifold.
		\item $\mathscr{A}_1$ is an orbifold and $X_0$ is not exceptional.
		\item $\Aun$ is finite and $X_0$ is not exceptional.
		\item $\Azero$ is trivial and $X_0$ is not exceptional.
	\end{enumerate}
\end{theorem}

Remark \ref{precisestatement} applies also here.

\begin{proof}
	By Theorem \ref{finitethm}, $\mathscr{T}(M,V)$ has finite isotropy groups if and only if $\mathscr{A}_1$ has finite isotropy groups. Now if $X_0$ is exceptional the finite group $\sharp\mathcal T_V$ is not the stabilizer of the base structure $X_0$, since the exceptional components do not yield automorphisms at $0$. Also, if $X_0$ is not exceptional, then Kur=Teich by definition. This proves that i) and ii) are equivalent. Then, ii) and iii) are equivalent by Corollary \ref{thmuniversal1}. Finally, iii) and iv) are equivalent because of the fact that $\Azero$ has finite index in $\Aun$ in the K\"ahler case.
\end{proof}

\begin{remark}
	\label{fakeorbifold}
	Consider the case of compact complex tori. Then $\Azero$ is not trivial, since it contains the translations. So neither $\mathscr{T}(M,V)$ nor $\mathscr{A}_1$ is an orbifold, since their isotropy groups are not finite. However, if we forget about the stack structure, the Teichm\"uller space is naturally a complex manifold. Indeed, roughly speaking, the stack is obtained from this complex manifold by attaching a group of translations to each point. More precisely, it is the universal family of tori, see \cite{LMStacks}, Example 13.1.
\end{remark}

\section{The non-K\"ahler case}
\label{sectionNK}
In this section, we briefly investigate the non-K\"ahler case. 

The atlases of the Kuranishi stacks and of $\mathscr{T}(M,V)$ in Sections \ref{Kurstack} and \ref{atlas} are still valid. However, Proposition \ref{Liebprop} and Corollaries \ref{interlemma}, \ref{2ndcor}, \ref{3rdcor} are no more correct. As a consequence, there may be a countable number of connected components in $\mathcal T_V$; and this number is not stable when restricting to smaller neighborhoods $V$. There may be wandering components. Indeed, taking $V'\subset V$, there is a natural injection of the set of connected components of $\mathcal T_V$ into the set of connected components of $\mathcal T_{V'}$, but it is possible that for a basis of neighborhoods $V\supset V_1\supset V_2\hdots$, all the injection maps are strict. So sequences \eqref{phin} may jump from one component to another and exceptional components do not capture all sequences \eqref{phin}.

For this reason, we have to deal with the target germification of $\mathscr{T}(M,V)$ and it is not even clear that there is a reasonable sense to the assertion: the inclusion of $\mathscr{A}_0$ and $\mathscr{A}_1$ in $\mathscr{T}(M,V)$ is an analytic morphism of stacks with countable fiber.

Also recall that cycles do not always converge in the non-K\"ahler setting, hence case iii) at the end of Section \ref{localTeich} really occurs. So we add the definitions
\begin{definition}
	We say that a component of $\Gamma$ is {\slshape adherent} to $X_0$ if its projection onto $K_0\times K_0$ does not contain $(0,0)$ but is adherent to it. We say that a sequence of components of $\mathscr{C}$ is {\slshape wandering} at $X_0$ if none of them is adherent to $(0,0)$, but their union is.  
\end{definition}
and 
\begin{definition}
		\label{defexceptionalg}
	A point $X_0$ of $\mathscr{T}(M)$ is called {\slshape adherent} if $\Gamma$ has a component adherent to it.
	It is {\slshape wandering} if $\mathscr{C}$ has a wandering sequence at $X_0$.
\end{definition}
We have
\begin{lemma}
	A compact complex manifold satisfies Kur=Teich if and only if it is neither an exceptional, an adherent nor a wandering point of the Teichm\"uller stack.
\end{lemma}
Theorem \ref{2ndmainthm} is no more correct and we do not know if the closure of exceptional, adherent and wandering points has an analytic structure. The main difficulty to argue as in the proof of Theorem \ref{2ndmainthm} is that the map $p$ is no more proper, hence does not send constructible sets onto copnstructible sets. Note that, since we now deal with spaces $\mathcal{T}_V$ with a possibly countable number of connected components, the best result we can hope for would be that the closure of exceptional, adherent and wandering points of the Teichm\"uller stack form a countable union of strict analytic substacks of $\mathscr{T}(M)$.

Finally, still due to the wandering components, Theorem \ref{mainorbifoldthm} is no more correct in the non-K\"ahler setting and we do not have any equivalent.

\section{Open Problems}
There are many open problems around this. We list some of them in this section.
Firstly, at the level of the Teichm\"uller stack/Kuranishi space,

\begin{problem}
	Find exceptional points.
\end{problem}
and, in the non-K\"ahler setting,
\begin{problem}
	Find exceptional, adherent and wandering points.
\end{problem}
\begin{remark}
	Consider the second Hirzebruch surface $\mathbb{F}_2$. It deforms onto $\mathbb{P}^1\times\mathbb{P}^1$. The automorphism $g$ that exchanges the two components of the product $\mathbb{P}^1\times\mathbb{P}^1$ does not extend as an automorphism of $\mathbb{F}_2$. So the corresponding graphs of biholomorphisms extend as singular cycles in $\mathbb{F}_2\times\mathbb{F}_2$. Hence there is a component of singular cycles in the cycle space of $\mathbb{F}_2\times\mathbb{F}_2$ that deform as automorphisms of the nearby structures. This really looks like an exceptional component. However, it is not, for $g$ is clearly not isotopic to the identity.
\end{remark}
Still in the non-k\"ahler setting,
\begin{problem}
	Prove/disprove that the closure of exceptional, adherent and wandering points of the Teichm\"uller stack form an at most countable union of strict analytic substacks of $\mathscr{T}(M)$.
\end{problem}
Then, at the level of compact complex manifolds,
\begin{problem}
Find a compact complex manifold $X_0$ with $\text{Aut}^1(X_0)$ having an infinite number of connected components.
\end{problem}

Such an example would automatically have $\sharp\mathcal T_V$ infinite. But there is no reason for them to be equal. Hence we ask 

\begin{problem}
Find a compact complex manifold $X_0$ with infinite $\sharp\mathcal T_V$ but $\text{Aut}^1(X_0)$ having a finite number of connected components.
\end{problem}

Both problems concern non-K\"ahler manifolds but in the K\"ahler setting, the difference between $\sharp\mathcal T_V$ and $\text{Aut}^1(X_0)/\text{Aut}^0(X_0)$ is also unknown. So we state our last problem

\begin{problem}
Find a compact K\"ahler manifold with $\sharp\mathcal T_V$ different from $\text{Aut}^1(X_0)/\text{Aut}^0(X_0)$. 
\end{problem}

\end{document}